\documentclass[11pt]{article}
\usepackage{amssymb,amsmath,amsfonts,mathrsfs,amsthm,epsfig,latexsym,color}
\usepackage{cite}
\usepackage{enumitem,geometry}
\usepackage{fourier}
\usepackage[all]{xy}
\usepackage{xcolor}

\makeatletter
\newcommand{\subsectionruninhead}{\@startsection{subsection}{2}{0mm}
	{-\baselineskip}{-0mm}{\bf\large}}
\newcommand{\subsubsectionruninhead}{\@startsection{subsubsection}{3}{0mm}
	{-\baselineskip}{-0mm}{\bf\normalsize}}
\makeatother

\newtheorem*{theorem*}{Theorem}
\newtheorem*{proof*}{Proof}
\newtheorem*{proposition*}{Proposition}
\newtheorem*{notation*}{Notation}
\newtheorem*{corollary*}{Corollary}
\newtheorem*{claim*}{Claim}
\newtheorem*{remark*}{Remark}

\newtheorem{proposition}{Proposition}[section]
\newtheorem{theorem}[proposition]{Theorem}
\newtheorem{corollary}[proposition]{Corollary}

\newtheorem{claim}[proposition]{Claim}

\theoremstyle{definition}
\newtheorem{definition}[proposition]{Definition}
\theoremstyle{remark}
\newtheorem{remark}[proposition]{Remark}

\newtheorem{question}[proposition]{Question}
\numberwithin{equation}{section}

\def\NN{\mathbb{N}}
\def\RR{\mathbb{R}}

\def\TT{\mathbb{T}}
\def\ZZ{\mathbb{Z}}

\def\e{{\varepsilon}}

\begin{document}
\title{\bf  Rigidity of center Lyapunov exponents for Anosov diffeomorphisms on 3-torus}
\author{Daohua Yu, ~ Ruihao Gu}
\date{\today}
\maketitle

\renewcommand{\abstractname}{Abstract}
\begin{abstract}
\normalsize
 Let $f$ and $g$ be two  Anosov diffeomorphisms on $\mathbb{T}^3$ with three-subbundles partially hyperbolic splittings where the weak stable subbundles are considered as center subbundles. Assume that $f$ is conjugate to $g$ and the conjugacy preserves the strong stable foliation, then their center Lyapunov  exponents of  corresponding periodic points coincide. This is the converse of the main result of  Gogolev and  Guysinsky in \cite{GG}.  Moreover, we get the same result for partially hyperbolic diffeomorphisms derived from Anosov on $\mathbb{T}^3$.
\end{abstract}

\pagenumbering{arabic}
\setcounter{page}{1}

\section{Introduction}
\qquad Let $M$ be a smooth closed Riemannian manifold. We say a diffeomrophism $f:M\to M$ is \textit{Anosov}, if there exists a $Df$-invariant continuous  splitting $TM=E^s\oplus E^u$ such that $Df$ is uniformly contracting on the stable bundle $E^s$ and uniformly expanding on the unstable bundle $E^u$. It is known that an Anosov diffeomorphism with one-dimensional stable or unstable bundle exists only on $d$-torus $\mathbb{T}^d$ ($d\geq 2$) \cite{F2, N}, in particular $\TT^3$ is the only  $3$-manifolds admitting Anosov diffeomorphisms.
Moreover an Anosov diffeomorphism $f$ on  $\mathbb{T}^d$ is always \textit{topologically conjugate} to its linearization $f_*:\pi_1(\mathbb{T}^d) \to \pi_1(\mathbb{T}^d)$ which  induces a toral Anosov automorphism  \cite{F1, M}, i.e., there exists a homeomorphism $h:\TT^d\to\TT^d$ homotopic to Id$_{\TT^d}$ such that $h\circ f=f_*\circ h$. However the conjugacy is usually only H${\rm \ddot{o}}$lder continuous. Indeed, if two Anosov diffeomorphisms are smoothly conjugate, then the derivatives of their return maps on corresponding periodic points are conjugate via the derivative of the conjugacy and more weakly their corresponding periodic Lyapunov exponents coincide. There are plenty of enlightening works researching the regularity of the conjugacy under the assumption of the same Lyapunov exponents, usually called \textit{rigidity}, e.g. \cite{Llave1992,GG,G2008,SY} in which the rigidity for Anosov diffeomorphisms with partially hyperbolic splittings has been studied extensively.

\begin{definition}\label{1 def ph}
	A diffeomorphism $f:M\to M$ is called \textit{partially hyperbolic}, if there is an invariant splitting for $Df$, $TM=E_f^{s}\oplus E_f^{c}\oplus E_f^{u}$, which satisfies $$||Dfv^{s}||<\tau(x)<||Dfv^{c}||<\mu(x)<||Dfv^{u}||, $$ for all $v^{\sigma}\in E_f^{\sigma}(x)$ with $||v^{\sigma}||=1$, where $\sigma=s,c,u$ and $0<\tau<1<\mu$ are two continuous functions on M.
\end{definition}

We call $E_f^{s}, E_f^{c}$ and $E_f^{u}$, the \textit{stable bundle}, the \textit{center bundle} and  the \textit{unstable bundle}, repectively. It is  well  known that the stable/unstable bundle of a partially hyperbolic diffeomorphism is always uniquely integrable and the integral manifolds form a foliation on $M$\cite{pesinbook}, called \textit{stable/unstable foliation} and denoted by $\mathcal{F}^{s/u}_f$.   However the center bundle in general is not integrable, see \cite{HHU} for a counterexample on $\TT^3$. A partially hyperbolic diffeomorphism $f$ is called \textit{dynamically coherent}, if there are two $f-$invariant foliations  denoted by $\mathcal{F}_f^{cu}$ and $\mathcal{F}_f^{cs}$ tangent  respectively to $E_f^{cu}:=E_f^{c}\oplus E_f^{u}$ and $E_f^{cs}:=E_f^{s}\oplus E_f^{c}$. It is clear that if $f$ is dynamically coherent, then $\mathcal{F}_f^{c}:=\mathcal{F}_f^{cu}\bigcap \mathcal{F}_f^{cs}$ is a $f-$invariant foliation tangent to $E^c_f$.  In particular, for the case of $||Dfv^{c}||<1$ for all unit $v^{c}\in E_f^{c}$, we also call  $E^c_f$ the \textit{weak stable bundle}, $E_f^s$ the \textit{strong stable bundle} and the corresponding foliation $\mathcal{F}^s_f$ is called \textit{strong stable foliation}.

  In \cite{GG}, Gogolev and  Guysinsky considered  the local rigidity of an Anosov automorphism $L$ on $\TT^3$ with partially hyperbolic splitting.
  For any two partially hyperbolic Anosov diffeomorphisms $f$ and $g$ on $\TT^3$ which are conjugate via $h$, we say $f$ and $g$ have the same \textit{center periodic data}, if their corresponding, by $h$, periodic points have the same  Lyapunov exponents on  center bundles.
    The main technical lemma in \cite{GG}  is that for $L$ with two-dimensional stable bundle, if  its $C^1$-perturbations $f$ and $g$ have the same center periodic data,  then their conjugacy $h$ \textit{preserves the strong stable foliation}, i.e. $h\big( \mathcal{F}^s_f(x) \big)= \mathcal{F}^s_g\big(h(x)\big)$ for all $x\in\TT^3$.  In this paper, we show the converse of this property.

\begin{theorem}\label{1 thm Anosov}
	Let $f,g:\TT^3\to\TT^3$ be two $C^{1+\alpha}$-smooth partially hyperbolic  Anosov diffeomorphisms  whose center bundles are weak stable.  Assume that $f$ is conjugate to $g$ via  a homeomorphism $h$. If $h$ preserves the strong stable foliations, then $f$ and $g$ have the same center periodic data.
\end{theorem}

An interesting corollary of Theorem \ref{1 thm Anosov} is that the topological conjugacy preserving the strong stable foliation implies it is smooth along  the center (weak stable) foliation. Here the conjugacy being \textit{$C^1$-smooth along the center foliation} is defined as the derivative along each center leaf being continuous with respect to the topology of the whole manifold. Note that in our case, the center  bundle is integrable  and the conjugacy preserves the center foliation \cite{Po}, also see Remark \ref{2 rmk Anosov conjugacy preserves center}.
\begin{corollary}\label{1 cor}
	Let $f,g:\TT^3\to\TT^3$ be two $C^{1+\alpha}$-smooth partially hyperbolic  Anosov diffeomorphisms  whose center bundles are weak stable.   Assume that $f$ is conjugate to $g$ via  a homeomorphism  $h$. Then $h$ preserves the strong stable foliation, if and only if, $h$ is $C^1$ along the center foliation.
\end{corollary}

\begin{remark}\label{1 rmk smooth}
One can get	Corollary \ref{1 cor} from Theorem \ref{1 thm Anosov} and \cite[Lemma 5 and Lemma 6]{GG} which state that the same center periodic data of $f$ and $g$  implies that $h$ is smooth along the center foliation and $h$ preserves the strong stable foliation. We mention that the assumption of local perturbation in \cite{GG} is just for getting that the center bundle of $f$ is integrable and preserved by $h$.  As mentioned above, we now have these two properties. For the inverse conclusion, for any periodic point $p$ with period $n$, since $h$ is smooth along the center foliation, we can take the directional derivative of $f^n=h^{-1}\circ g^n\circ h$ in the direction of $E^c$, we can get that $\lambda^c(p,f)=\lambda^c(h(p),g)$. From the conclusion of \cite{GG}, we can know that $h$ preserves the strong stable foliation.
\end{remark}

We are also  curious about the higher-dimensional case of Theorem \ref{1 thm Anosov} and Corollary \ref{1 cor}, since Gogolev extended the result of local rigidity in \cite{GG} to the higher-dimensional case \cite{G2008}. For convenience, we state this question in Section  \ref{sec: 3 torus}, see Question \ref{question}.

\vspace{2mm}

In this paper, we also consider partially hyperbolic diffeomorphisms on general closed Riemannian manifolds  with integrable one-dimensional center bundles and with accessibility. We call a  partially hyperbolic diffeomorphism $f:M\to M$ is \textit{accessible}, if  any two points on $M$ can be connected by curves each of which is  tangent to $E^s_f$ or $E^u_f$. It is well known that accessibility is  $C^r$-dense and $C^1$-open among $C^r\ (r\geq 1)$ partially hyperbolic diffeomorphisms with one-dimensional center bundles \cite{D03,DW03,BHHTU,HHU08}.

Let $f,g:M\to M$ be two  partially hyperbolic diffeomorphisms and dynamically coherent.  We call  $f$ is \textit{fully conjugate} to $g$, if there exists a homeomorphsim $h:M\to M$ such that $f$ is conjugate to $g$ via $h$ and $h$ preserves the stable foliations, unstable foliations and center foliations, respectively, i.e., $$h\big( \mathcal{F}^{\sigma}_f(x) \big)= \mathcal{F}^{\sigma}_g\big(h(x)\big), \ \   \forall x\in M\ \ {\rm and}\ \  \sigma=s,c,u.$$ It is clear that if $f$ is accessible and fully conjugate to $g$, then $g$ is also accessible.  Now, we can extend Theorem \ref{1 thm Anosov} to the following one in which the dimension of the manifold $M$ can be bigger than $3$. We mention that  $h$ does not need to be homotopic to identity in Theorem \ref{1 thm ph}.

\begin{theorem}\label{1 thm ph}
Let $f,g:M\to M$ be two $C^{1+\alpha}$-smooth partially hyperbolic diffeomorphisms with one-dimensional center bundles and dynamically coherent. If $f$ is accessible and  fully conjugate to $g$ via a homeomorphism $h$,  then $h$ is $C^1$ along the center foliation.
\end{theorem}

 As a corollary,  we  apply Theorem \ref{1 thm ph} to  partially hyperbolic diffeomorphisms  \textit{derived from Anosov} on $\mathbb{T}^3$. We call $f:\mathbb{T}^3\to \TT^3$ is derived from Anosov, also called a \textit{DA} diffeomorphism, i.e.,  its linearization $f_*:\pi_1(\TT^3)\to \pi_1(\TT^3)$  induces an Anosov automorphism.  We mention that the partially hyperbolic diffeomorphism derived from Anosov is one of the three types of partially hyperbolic diffeomorphisms on $3$-manifolds with solvable fundamental group\cite{HP2015}.

\begin{corollary} \label{1 cor DA}
	Let $f, g :\mathbb{T}^{3}\to \mathbb{T}^{3}$ be two $C^{1+\alpha}$-smooth partially hyperbolic diffeomorphisms  and  conjugate  via  a homeomorphism  $h$. Suppose  the linearization of $f$  is an Anosov automorphism with one-dimensional unstable bundle.  If $h$ preserves the stable foliation, then $h$ is $C^1$ along the center foliation.
\end{corollary}

It is clear that Theorem \ref{1 thm Anosov} is a special case of Corollary \ref{1 cor DA}, as
we can take the directional derivative of $f^n=h^{-1}\circ g^n\circ h$ in the direction of $E^c$ at periodic point $p$ with period $n$ to get the same center periodic data.
Another special case of Corollary \ref{1 cor DA} is that $g=f_*$ which has been studied in \cite{GanS}. We  refer readers to \cite{gogolevshi} for higher-dimensional case under the assumption of $g=f_*$.

\section*{Acknowledgements}
\addcontentsline{toc}{section}{Acknowledgements}
\qquad We are grateful for the valuable communication and suggestions from Shaobo Gan and Yi Shi.
Thanks for the comments from the annoymous referees about Remark 2.7 and Remark 3.4.
The authors were partially supported by National Key R\&D
Program of China (2021YFA1001900).

\section{Rigidity of center Lyapunov exponents in the accessible case}

\qquad In this section, we prove Theorem \ref{1 thm ph}. First of all, we recall some basic notions  and useful properties of partially hyperbolic diffeomorphisms.
Let $M$ be a smooth cloesd Riemannian manifold and $f:M\to M$ be a  diffeomorphism admitting partially hyperbolic splitting $TM=E_f^{s}\oplus E_f^{c}\oplus E_f^{u}$ with dim$E^c_f=1$.
The \textit{Lyapunov exponent} on $E^{c}_f$ at point $x$, if it exists, is denoted by $\lambda^{c}(x,f)$ and defined as
$$\lambda^{c}(x,f)=\lim_{n\to +\infty} \frac{1}{n}\ {\rm log}\|Df^n|_{E^c_f(x)}\|.$$
Assume further that $f$ is dynamically coherent. Define the \textit{local leaf} with size $\delta>0$ by $$\mathcal{F}^{\sigma}_f(x,\delta):=\big\{y\in \mathcal{F}^{\sigma}_f(x)\ | \ d_{\mathcal{F}^{\sigma}_f}(x,y)< \delta \big\},$$
 where $\sigma=s,c,u,cs,cu$ and $d_{\mathcal{F}^{\sigma}_f}(\cdot,\cdot)$ is the metric on $\mathcal{F}^{\sigma}_f$ induced by the Riemannian metric on the base space. By coherence, the local stable/unstable foliation $\mathcal{F}^{s}_f\big/ \mathcal{F}^{u}_f$ induces \textit{holonomy maps}  restricted on $\mathcal{F}^{cs}_f \big/ \mathcal{F}^{cu}_f$ as follow,
 \begin{align*}
 	{\rm Hol}_{f,x,y}^{s/u}: \mathcal{F}^c_f(x,\delta_1)&\longrightarrow \mathcal{F}^c_f(y,\delta_2),\\
 	 z &\longmapsto \mathcal{F}^{c}_f(y,\delta_2)\cap \mathcal{F}^{s/u}_f(z,R),
 \end{align*}
 such that 	${\rm Hol}_{f,x,y}^{s/u}$ is a homeomorphism for some $x\in M$ and $y\in\mathcal{F}_f^{s/u}(x)$ with constants   $\delta_1>0,\delta_2>0$ and $R>0$ which rely on the choice of $x,y$.

 Since one-dimensional center bundle automatically implies the center bunching condition in the classical work  \cite{PSW}, we have the following proposition adapted to our case. We also refer to \cite[Theorem 7.1]{pesinbook} for  this property in the view of absolutely continuous holonomy map.
  \begin{proposition}[\cite{PSW}]\label{2 prop smooth holonomy}
  	Let $f:M\to M$ be a $C^{1+\alpha}$-smooth partially hyperbolic diffeomorphism with one-dimensional center bundle. If $f$ is dynamically coherent, then the local unstable and local stable holonomy maps restricted respectively on $\mathcal{F}_f^{cu}$ and $\mathcal{F}_f^{cs}$ are uniformly $C^{1}$-smooth.
  \end{proposition}
\begin{remark}
 The uniformly $C^1$-smooth holonomy maps in Proposition \ref{2 prop smooth holonomy} means that the derivative $D{\rm Hol}^{s/u}_{f,x,y}$ is continuous with respect to points $x,y\in M$. We refer to \cite{PSW2004} for delicate analysis for the partial derivatives of the (un)stable bundle  with respect to center bundle.
\end{remark}

   Recall that we say $f$ is accessible, if  any two points on $M$ can be connected by curves each of which is  tangent to $E^s_f$ or $E^u_f$. In fact, these curves have uniform length and number \cite[Proposition 4]{D03}. For convenience of readers, we state it as follow.

     \begin{proposition}[\cite{D03}]\label{2 prop finite accessibility}
     Let $f:M\to M$ be a $C^1$-smooth partially hyperbolic diffeomrophism and accessible. Then there exists $N\in\NN$ and $R>0$ such that  any two points $x,y\in M$ can be connected by at most $N$ curves each of which has length at most $R$ and tangent to $E^s_f$ or $E^u_f$.
     \end{proposition}

   Now we can prove Theorem \ref{1 thm ph}.

    \begin{proof}[Proof of Theorem \ref{1 thm ph}]
    Let $h:M\to M$ be a full conjugacy	between $f$ and $g$ such that
     $$h\big( \mathcal{F}^{\sigma}_f(x) \big)= \mathcal{F}^{\sigma}_g\big(h(x)\big), \ \   \forall x\in M\ \ {\rm and}\ \  \sigma=s,c,u.$$
    Then for any point $x\in M$, $h$ restricted on the local center leaf $\mathcal{F}_f^c(x,\delta)$ of $x$ is a homeomorphism onto the image. For simplify, we denote it by
    $$h_{c,x}:\mathcal{F}_f^c(x,\delta)\to \mathcal{F}_g^c\big(h(x)).$$
    Fix a point $z_0\in M$. Since $\mathcal{F}_f^c(z_0,\delta)$ and $\mathcal{F}_g^c\big(h(z_0)\big)$ are one-dimensional $C^1$-smooth embedded submanifolds, we get that $h_{c,z_0}$ is monotonic. Therefore, there exists a point $p\in \mathcal{F}_f^c(z_0,\delta)$ such that $h_{c,z_0}$ is differentiable at $p$, then $h_{c,p}$ is differentiable at $p$.  By the accessibility and the $C^1$-smooth holonomy maps, we can get that $h_{c,x}$ is differentiable at  $x$ for all $x\in M$.  More precisely, we have the following claim.

    \begin{claim}\label{3 claim holonomy get Dhx=0}
    For any $x\in M$, $h_{c,x}$ is differentiable at $x$.
    \end{claim}

    \begin{proof}[Proof of Claim \ref{3 claim holonomy get Dhx=0}]
    	Since $f$ is accessible, in particular by Proposition \ref{2 prop finite accessibility}, there exists $N$ and $R$ such that for any $x\in M$, there exists at most $N$ local stable and unstable manifolds with length at most $R$ to connect $x$ and $p$. Hence, we can assume that there exist $p=x_0, x_1,...,x_n=x \ (n\leq N)$  such that
    	$$x_{i}\in \mathcal{F}^{s/u}_f(x_{i-1},R)\quad {\rm and} \quad x_{i+1}\in \mathcal{F}^{u/s}_f(x_i,R), \quad \forall 1\leq i\leq n-1. $$
    	Let $y_i=h(x_i)$ for all $0\leq i\leq n$. Since $h_{c,p}$ is differentiable at $p=x_0$, it suffices to prove that if $x_i$ is a differentiable point of $h_{c,x_i}$, then $x_{i+1}$ is a differentiable point of $h_{c,x_{i+1}}$.

    	Without loss of generality, we assume that $x_{i+1}\in \mathcal{F}^{s}_f(x_i,R)$.
    	Since $h$ preserves all invariant foliations of $f$ and $g$, for any $x_{i+1}$, there exists $\delta>0$ such that
    	$$ h_{c,x_{i+1}}(z)={\rm Hol}^s_{g,y_i,y_{i+1}} \circ h_{c,x_i} \circ {\rm Hol}^s_{f,x_{i+1},x_i}(z), \quad \forall z\in \mathcal{F}^c_f(x_{i+1},\delta).$$
    	See Figure 1. Notice that
    \begin{align}
     Dh_{c,x_{i+1}}(x_{i+1})=D{\rm Hol}^s_{g,y_i,y_{i+1}}(y_i)\circ Dh_{c,x_i}(x_i)\circ D{\rm Hol}^s_{f,x_{i+1},x_i}(x_{i+1}). \label{eq. Dh_c}
    \end{align}
    Since both Hol$^s_{f,x_{i+1},x_i}$ and Hol$^s_{g,y_i,y_{i+1}}$ are $C^1$-smooth within size $R>0$ (see Proposition \ref{2 prop smooth holonomy}), we get the proof of this claim.
    \end{proof}

    By Claim \ref{3 claim holonomy get Dhx=0}, we can also get that for any $x\in M$, $Dh_{c,x}(x)\neq 0$. Indeed, if there exists $q\in M$ such that   $Dh_{c,q}(q)= 0$, it follows from the proof of Claim \ref{3 claim holonomy get Dhx=0}  that $Dh_{c,x}(x)= 0$ for any $x\in M$. It contradicts with the fact that $h$ is a homeomorphism restricted on $\mathscr{F}^c_f(p)$. Then we get that $h_c$ is differentiable.

  \begin{figure}[htbp]
  	\centering
  	\includegraphics[width=10cm]{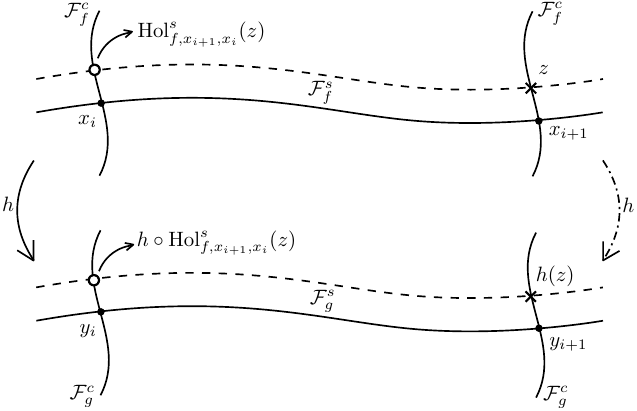}
  	\caption{The conjugacy restricted on $ \mathcal{F}^c_f(x_{i+1},\delta)$ is commutative with holonomy maps.}	
  \end{figure}

  For proving that $h_c$ is actually $C^1$-smooth, we need more precise property on accessibility. Since $f$ is accessible with one-dimensional center bundle, there are two unstable disks $U_1$ and $U_2$ being \textit{skew} \cite{HHU08,HU}, that is
 \begin{enumerate}
 	\item There exists a $( {\rm dim}E^u+1)$-dimensional $cu$-disk $V$ containing $U_2$ (note that $f$ is coherent).
 	\item $U_2$ separates $V$ into two connected components.
 	\item $U_1$ and $V$ define the holonomy maps Hol$^s: U_1\to V$ as Hol$^s(x)=\mathcal{F}^s_{\delta}(x)\cap V$ for small $\delta$.
 	\item Hol$^s(U_1)$ intersects both connected components of $V\backslash U_2$.
 \end{enumerate}

  \begin{figure}[htbp]
	\centering
	\includegraphics[width=8cm]{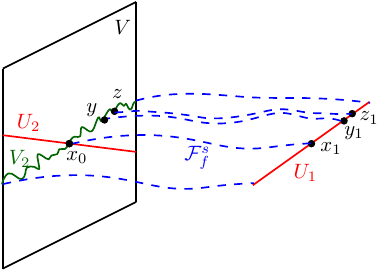}
	\caption{The $su$-holonomies vary continuously on $V_2$ in $C^1$-topology.}	
\end{figure}

\begin{claim}\label{claim Dhc continuous on V}
	For any $z\in V$, the derivative $Dh_{c,z}$ is continuous at $z$ restricted on $V$.
\end{claim}

\begin{proof}[Proof of Claim \ref{claim Dhc continuous on V}]

 Denote $V\cap {\rm Hol^s}(U_1)=V_2$. Take $x_0\in V_2\cap U_2$.
 Denote $\big({\rm Hol}^s\big)^{-1}(x_0)=x_1\in U_1$. For any $y\in V_2$, denote $\big({\rm Hol}^s\big)^{-1}(y)=y_1\in U_1$. Then $x_0$ and $y$ can be connected by $3$-legs of stable and unstable leaves through points $x_1$ and $y_1$.   Note  that if $z\in V_2$ is close to $y$, then the $3$-legs from $z$ to $x_0$ is close (in $C^1$-topology) to the $3$-legs from $y$ to $x_0$,  since  local stable and unstable manifolds vary continuously in $C^1$-topology (see Figure 2).  It follows from \eqref{eq. Dh_c} that
  \begin{align*}
 	Dh_{c,y}(y)=D{\rm Hol}^{sus}_{g,h(x_0),h(y)}(h(x_0))\circ Dh_{c,x_0}(x_0)\circ D{\rm Hol}^{sus}_{f,y,x_0}(y),
 \end{align*}
 and
 \begin{align*}
 	Dh_{c,z}(z)=D{\rm Hol}^{sus}_{g,h(x_0),h(z)}(h(x_0))\circ Dh_{c,x_0}(x_0)\circ D{\rm Hol}^{sus}_{f,z,x_0}(z),
 \end{align*}
where ${\rm Hol}^{sus}_{f,y,x_0}= {\rm Hol}^s_{f,x_1,x_0}  \circ {\rm Hol}^u_{f,y_1,x_1} \circ {\rm Hol}^s_{f,y,y_1}$ and the notation ${\rm Hol}^{sus}_{g,h(x_0),h(y)}$ is defined in the same sense.  By the uniform $C^1$ regularity of holonomy maps (see Proposition \ref{2 prop smooth holonomy}), we have that $Dh_{c,y}(y)$ varies continuously at $y$ restricted on $V_2$.

 Then we will "transfer" the  continuity of $Dh_c$ from $V_2$ to $V$ via the unstable foliation. Note that by the local product structure of foliation $\mathcal{F}^{cs}_f$ and $\mathcal{F}^u_f$, one has that every local center leaf intersects with $V_2$ once.  However, a local unstable leaf may intersect with $V_2$ infinitely many times (see Figure \ref{cont}). For convenience, we assume that $$V=\bigcup_{x\in V_2}\mathcal{F}^u_{f,{\rm loc}}(x),$$
 and $V$ is open in the topology of a  local $cu$-leaf. For any $z\in V$, denote $S(z)=\mathcal{F}^u_{f,{\rm loc}}(z)\cap V_2$.


  \begin{figure}[htbp]
 	\centering
 	\includegraphics[width=7cm]{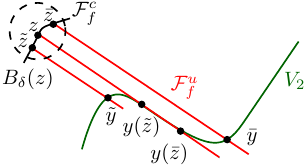}\\
 	\caption{Different $u$-paths to calculate the derivative $Dh_c$ on a neighborhood of $z$ in $V$.}\label{cont}
 \end{figure}

It is clear that for any $\epsilon>0$, there exists $0<\delta<\epsilon$ such that $$ B_\delta(z)\subset\bigcup_{s\in S(z)}\mathcal{F}_{f,{\rm loc}}^{u}\big(B_\epsilon(s)\big),$$
 where $B_\epsilon(s)$ is an $\e$-open ball in $V$.
 Then for any $\bar{z},  \tilde{z} \in B_{\delta}(z)$, there exist two points $\bar{y} \in \mathcal{F}_{f,{\rm loc}}^u(\bar z)\cap V_2$ and $\tilde{y} \in \mathcal{F}_{f,{\rm loc}}^u(\tilde{z})\cap V_2$ such that there are  points $ y(\bar z)\in S(z)$ and $y(\tilde{z})\in S(z)$ such that $\bar{y}\in B_\epsilon(y(\bar z))$ and $\tilde{y}\in B_\epsilon(y(\tilde z))$ (see Figure \ref{cont}).  Note that
 \begin{align*}
 	Dh_{c,z}(z)=D{\rm Hol}^u_{g,h(y(\bar z)),h(z)}\big(h(y(\bar z))\big)\circ Dh_{c,y(\bar z)}\big(y(\bar z)\big)\circ D{\rm Hol}^u_{f,z,y(\bar z)}(z),
 \end{align*}
 and
 \begin{align*}
 	Dh_{c,\bar z}(\bar z)=D{\rm Hol}^u_{g,h(\bar y),h(\bar z)}\big(h(\bar y)\big)\circ Dh_{c,\bar y}(\bar y)\circ D{\rm Hol}^u_{f,\bar z,\bar y}(\bar z).
 \end{align*}
Since  holonomy maps are uniformly $C^1$-smooth  and $Dh_{c}$ is continuous on $V_2$, we have that $Dh_{c,z}(z)$ is close to $Dh_{c,\bar{z}}(\bar{z})$.
Similarly, we can use the holonomy maps given by  the local unstable leaves of $\tilde{y}$ and $y(\tilde{z})$ to control the variation between $Dh_{c,z}(z)$ and $Dh_{c,\tilde{z}}(\tilde{z})$.
Then we have that $Dh_{c,z}(z)$ varies continuously at $z$ restricted on $V$.
\end{proof}

\begin{claim}\label{claim Dhc continuous on M}
For any $w\in M$, the derivative $Dh_{c,w}$ is continuous at $w$ on $M$.
\end{claim}

\begin{proof}[Proof of Claim \ref{claim Dhc continuous on M}]
	For any point $w\in M$, there is an $su$-path from $w$ to some point $z\in V$, since $f$ is accessible. Thus there exists a neighborhood $B(w)$ of $w$ in $M$ such that for any point $w'\in B(w)$ there is an $su$-path from $w'$ to $z'\in V$ close to the $su$-path from $w$ to $z$ in the sense of $C^1$-topology (see Figure 4). Then we have that
\begin{align*}
	Dh_{c,w}(w)=D{\rm Hol}_{g,h(z),h(w)}(h(z))\circ Dh_{c,z}(z)\circ D{\rm Hol}_{f,w,z}(w),
\end{align*}
and
\begin{align*}
	Dh_{c,w'}(w')=D{\rm Hol}_{g,h(z'),h(w')}(h(z'))\circ Dh_{c,z'}(z')\circ D{\rm Hol}_{f,w',z'}(w'),
\end{align*}
where ${\rm Hol}_{f,w,z}$ is the holonomy map (given by some compositions of Hol$^s$ and Hol$^u$) of a fixed $su$-path from $w$ to $z$ and $D{\rm Hol}_{f,w',z'}$ is the honolomy map of the $su$-path from $w'$ to $z'$ which is close (in $C^1$-topology) to the fixed $su$-path (from $w$ to $z$).  And ${\rm Hol}_{g,h(z),h(w)}$ is given by the  image of the  fixed $su$-path (from $w$ to $z$) via the conjugacy $h$.
This implies that  $Dh_{c,w}(w)$ is continuous for any $w\in M$ since we already have that $Dh_c$ is continuous on $V$. 
\end{proof}

  \begin{figure}[htbp]
	\centering
	\includegraphics[width=7cm]{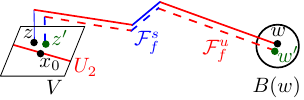}
	\caption{The $su$-holonomies vary continuously in $C^1$-topology.}	
\end{figure}

Hence $h$ is $C^1$-smooth along the center foliation.
\end{proof}

\begin{remark}\label{h_c-$C^1$}
To prove the  $C^1$-regularity of  $h_c$, there is another approach which needs more subtle functional analysis on the derivatives. From Theorem 7.3 in \cite{O}, we can know that if a function is differentiable everywhere, then the derivative is continuous on a residual set. When  Claim \ref{3 claim holonomy get Dhx=0} is gotten, there exists a point $z\in M$ such that $Dh_{c,z}(z)$ is $C^0$ restricted on $\mathcal{F}^c_f(z)$.  Just like the proof of Claim \ref{claim Dhc continuous on M},  we can get that $Dh_{c,x}(x)$ is $C^0$ at any point $x\in M$ with respect to the topology of $M$ by the equation (\ref{eq. Dh_c}), since the  holonomies  are $C^1$ and the systems are accessible.
\end{remark}

   \section{Anosov diffeomorphisms on $3$-torus}\label{sec: 3 torus}

\qquad In this section, we prove Corollary \ref{1 cor DA} and hence we can get the proof of Theorem \ref{1 thm Anosov}. In fact, we  will reduce Corollary \ref{1 cor DA} to the accessible case Theorem \ref{1 thm ph}.  At the end of this section, we state a question for higher-dimensional torus. Note that  in this section we always consider the topological conjugacy being homotopic to identity.

 Let  $f:\TT^3\to \TT^3$ be a partially hyperbolic diffeomorphism homotopic to an Anosov automorphism $A:\TT^3\to\TT^3$. It has been proved in \cite{Po} that  $f$ is dynamically coherent. Moreover,  $A$ also admits a partially hyperbolic splitting  $$T\TT^3=E_A^s\oplus E_A^c\oplus E_A^u.$$ We call such $A:\TT^3\to\TT^3$  is  \textit{center contracting}, if the modulus of the eigenvalue of $A\in {\rm GL}_3(\ZZ)$ corresponding to $E^c_A$ is smaller than $1$.  Further, one has  the following  proposition.  Denote  by $\widetilde{\mathcal{F}}_{*}^{\sigma}\ (*=f,A,\ {\rm and}\  \sigma=s,c,u,cs,cu)$ the  liftings of $\mathcal{F}_{*}^{\sigma}$ by the natural projection $\pi:\RR^3\to \TT^3$.

\begin{proposition}[\cite{F2, Po}]\label{2 prop semi-conjugate}
	Let $f:\TT^3\to \TT^3$ be a partially hyperbolic diffeomorphism homotopic to a center contracting Anosov automorphism $A:\TT^3\to\TT^3$. Let  $F$ be a lifting of $f$ by $\pi$. Then there exists a uniformly continuous surjection $H_f:\RR^3\to\RR^3$ (called semi-conjugacy) such that
	\begin{enumerate}
		\item $H_f\circ F=A\circ H_f$ and there is $C>0$ satisfying $d_{C^0}(H_f, {\rm Id}_{\RR^3})<C$. And such $H$ is unique.
		\item  $H_f(x+n)=H_f(x)+n$ for all $x\in\RR^3$ and $n\in\ZZ^3$ and hence $H$ can descend to $\TT^3$ denoted by $h_f:\TT^3\to \TT^3$ such that  $h_f$ is homotopic to Id$_{\TT^3}$ and $h_f\circ f=A\circ h_f$.
		\item $H_f:\widetilde{\mathcal{F}}_{f}^{u}(x)\to \widetilde{\mathcal{F}}_{A}^{u}(H_f(x))$ is a homeomorphism, for all $x\in\RR^3$.
		\item $H_f\big( \widetilde{\mathcal{F}}_{f}^{\sigma}  (x) \big) =\widetilde{\mathcal{F}}_{A}^{\sigma}(H_f(x))$, for all $x\in\RR^3$ and $\sigma=c,cs,cu$.
	\end{enumerate}
\end{proposition}

\begin{remark}\label{2 rmk Anosov conjugacy preserves center}
	If $f:\TT^3\to\TT^3$ is Anosov, then the semi-conjugacy $h_f$ given by Proposition \ref{2 prop semi-conjugate} is in fact a conjugacy. It follows from the fourth item that $h_f$ preserves the center foliation and the restriction of $h_f$ on each center leaf is a homeomorphism.
\end{remark}

  We call $f$ is $su$-\textit{integrable}, if there exists a $({\rm dim} E_f^s+{\rm dim} E_f^u)$-dimensional foliation denoted by $\mathcal{F}^{su}_f$ tangent to $E^s_f\oplus E^u_f$.
It is clear from the definitions that if $f$ is accessible, then it is not $su$-integrable. Conversely,
\cite{HU} proved that for a conservative partially hyperbolic diffeomorphism $f:\TT^3\to\TT^3$   derived from Anosov,   if $f$  is not accessible, then it is $su$-integrable and conjugate to its linearization. We mention that the conservative condition can be replaced by assuming the non-wandering set $\Omega(f)=\TT^3$.
Moreover without the conservative condition, combining the main results in \cite{GanS} and  \cite{HS}, one has the following dichotomy (Theorem \ref{2 thm dichotomy}, also see \cite[Corollary 1.4]{HS}).
We also refer to \cite{GRZ} for the case of Anosov diffeomorphism $f$ on $\TT^3$ in which one has that the conjugacy between $f$ and its linearization $f_*$ preserves the strong stable foliation, if and only if, $f$ is $su$-integrable, if and only if, $f$ is not accessible.

\begin{theorem}[\cite{GanS,HS}]\label{2 thm dichotomy}
	Let $f:\TT^3\to \TT^3$ be a $C^{1+\alpha}$-smooth partially hyperbolic diffeomorphism homotopic to an Anosov automorphism $A:\TT^3\to\TT^3$. Then
	\begin{itemize}
		\item either $f$ is accessible,
		\item or $f$ is Anosov and $\lambda^c(p,f)=\lambda^c(A)$, for all $p\in{\rm Per}(f)$.
	\end{itemize}
\end{theorem}

\begin{remark}
We mention here that we can get our Theorem \ref{1 thm Anosov} by combining Theorem \ref{1 thm ph} with the main results of  \cite{GanS} and  \cite{HU}. Indeed, let $f,g:\TT^3\to \TT^3$ satisfy the assumption of  Theorem \ref{1 thm Anosov}. Then by \cite{HU}, one has that $f$ and $g$ are simultaneously either  $su$-integrable, or  accessible. In the first case, the center periodic data of $f$ and $g$ are the same as one of their linearization respectively \cite{GanS}.  The other case is an immediate corollary of Claim \ref{3 claim holonomy get Dhx=0}.
\end{remark}

By Theorem \ref{2 thm dichotomy}, we will reduce Corollary  \ref{1 cor DA} to  Theorem \ref{1 thm ph}.

\begin{proof}[Proof of Corollary \ref{1 cor DA}]
	Let $f$ and $g$ satisfy  the conditions of Corollary \ref{1 cor DA}. Since $f$ is conjugate to $g$ via $h$  which is a homeomorphism homotopic to Id$_{\TT^3}$ such that $h\circ f= g\circ h$ and $f$ is derived from Anosov automorphism $A\in {\rm GL}_3(\ZZ)$. We have that $g_*=(h\circ f\circ h^{-1})_*=f_*=A$.
	
	We claim that $f$ is accessible if and only if $g$ is accessible. Indeed, let $F,G$ be any two liftings  of $f$ and $g$ by the natural projection $\pi$. Let $H_f$ and $H_g$ be the semi-conjugacy between $F$  with $A$  and $G$  with $A$ respectively given by Proposition \ref{2 prop semi-conjugate}.
	Let $h$ be a conjugacy between $f$ and $g$ and let $H$  be its lifting  with $H\circ F=G\circ H$. Note that $h$ is homotopic to Id$_{\TT^3}$, thus $H$ is uniformly bounded with Id$_{\RR^3}$. Since $H_f\circ H^{-1}$ is also a semi-conjugacy between $G$ and $A$ satisfying the first item of Proposition \ref{2 prop semi-conjugate}, we have $H_g=H_f\circ H^{-1}$ by the uniqueness. By Proposition \ref{2 prop semi-conjugate} again, one has that
	\begin{align}
		H_{*}\big( \widetilde{\mathcal{F}}_{*}^{\sigma} (x)  \big)= \widetilde{\mathcal{F}}_{A}^{\sigma} (H_{*}(x)),\   \ \forall x\in\RR^3,\ \  \sigma=c,u\  \ {\rm and}\  \ *=f,g. \label{eq. 2. H_f preserves center}
	\end{align}
	It follows that
	\begin{align}
		H\big( \widetilde{\mathcal{F}}_{f}^{\sigma} (x)  \big)=  \widetilde{\mathcal{F}}_{g}^{\sigma} \big(H(x)\big),\  \ \forall x\in\RR^3\  \ {\rm and}\ \  \sigma=c,u. \label{eq. 2. H preseves center}
	\end{align}
	In fact, if \eqref{eq. 2. H preseves center} is not true, then one can get a contradiction immediately with \eqref{eq. 2. H_f preserves center} and the assumption $H\big( \widetilde{\mathcal{F}}_{f}^{s} (x)  \big)=  \widetilde{\mathcal{F}}_{g}^{s} (H(x))$. In particular, we get that  $h\big(\mathcal{F}^{u/s}_f(x) \big)= \mathcal{F}^{u/s}_g(h(x))$ for all $x\in\TT^3$. Hence $f$ is accessible if and only if $g$ is accessible.
	
	If $f$ and $g$ are both accessible,  in particular, by \eqref{eq. 2. H preseves center}, we have that
	\begin{align*}
		h\big(\mathcal{F}^{\sigma}_f(x) \big)= \mathcal{F}^{\sigma}_g(h(x)), \ \ \forall x\in\TT^3\ \ {\rm and}\ \ \sigma=s,c,u,
	\end{align*}
   then we obtain Corollary \ref{1 cor DA} from Theorem \ref{1 thm ph}.

   If neither of them  is accessible, it follows from Theorem \ref{2 thm dichotomy} that   $\lambda^c(p,f)=\lambda^c(A)=\lambda^c(h(p),g)$, for all $p\in{\rm Per}(f)$ and hence $h$ is $C^1$ along the center foliation (see Remark \ref{1 rmk smooth}).
\end{proof}

\vspace{2mm}
 As mentioned before, we now state a question  related  to Theorem \ref{1 thm Anosov}  and  the local rigidity in \cite{G2008} of higher-dimensional torus.  For convenience, we first restate a main property in \cite{G2008}. Let $A:\TT^d\to\TT^d\ (d\geq 4)$ be an Anosov automorphism  with the \textit{finest dominated splitting on the weak stable bundle}, namely,  the stable bundle $L^s$ of $A$ admits
$$L^s=L^s_1\oplus L^{ws}=L^s_1\oplus L^s_2\oplus...\oplus L^s_k,$$
where $L^s_j \ (1\leq j\leq k)$ is $DA$-invariant, $L^s_i \ (2\leq i\leq k)$ is one-dimensional and there exists a constant $0<\lambda<1$ such that $\|DA|_{ L^s_i(x)}\|/ \|DA|_{L^s_{i+1}(x)}\|\leq \lambda$, for all $x\in\TT^d$ and all $1\leq i\leq k-1$.
Note that we can choose an appropriate $C^1$ neighborhood $\mathcal{U}$ of $A$ such that any $f\in\mathcal{U}$ is also Anosov diffeomorphism with finest dominated splitting on the weak stable bundle. In \cite{G2008}, Gogolev proved that there exists a $C^1$ neighborhood $\mathcal{U}$ of $A$ such that for any two $C^2$-smooth Anosov diffeomorphisms $f,g\in\mathcal{U}$ with finest dominated splittings on weak stable bundles $E^s_{\sigma}=E^s_1(\sigma)\oplus...\oplus E^s_k(\sigma)\ ,(\sigma=f,g)$,  if $f$ and $g$ have the same \textit{weak stable periodic data}, i.e. the periodic Lyapunov exponents on each weak stable bundles $E^s_i(f)\ (2\leq i\leq k)$ coincide with ones of $E^s_i(g)$, then the conjugacy between $f$  and $g$ preserves the strongest stable foliations corresponding to $E^s_1(f)$ and $E^s_1(g)$.  Hence we can also consider the converse of this property. More precisely, we state this question as  follow.
\begin{question}\label{question}
	Let $A:\TT^d\to\TT^d\ (d\geq4)$  be an Anosov automorphism with the finest dominated splittings on the weak stable bundle.  Is there a $C^1$ neighborhood $\mathcal{U}$ of $A$ such that for any two $C^{1+\alpha}$-smooth  Anosov diffeomorphisms $f,g\in\mathcal{U}$ with finest dominated splittings on weak stable bundles, if  the conjugacy between $f$ and $g$ preserves the strongest stable foliations, then $f$ and $g$ have the same weak stable periodic data?
\end{question}
\begin{remark}
 We  refer again to \cite{gogolevshi} for a positive answer to Question \ref{question} under the assumption of $g=A$.
\end{remark}

\bibliographystyle{plain}

\bibliography{ref}

\flushleft{\bf Daohua Yu} \\
School of Mathematical Sciences, Peking University, Beijing, 100871,  China\\
\textit{E-mail:} \texttt{yudh@pku.edu.cn}\\

\flushleft{\bf Ruihao Gu} \\
School of Mathematical Sciences, Peking University, Beijing, 100871,  China\\
\textit{E-mail:} \texttt{rhgu@pku.edu.cn}\\

\end{document}